\documentclass[envcountsect,referee]{svjour3}

\smartqed  
\usepackage{graphicx}
\usepackage{latexsym,bm,amsmath,amssymb,mathrsfs}
\usepackage{rotating}
\usepackage{multirow,booktabs,cases}
\usepackage[misc]{ifsym}
\usepackage[style=1]{mdframed}
\usepackage{algorithm,algorithmic}
\usepackage[colorlinks=true]{hyperref}
\hypersetup{urlcolor=blue,citecolor=blue,linkcolor=blue}
\usepackage{epstopdf}

\def\[{\begin{equation}}
\def\]{\end{equation}}

\newtheorem{assum}{Assumption}[section]
\newtheorem{exam}{Example}[section]

\numberwithin{equation}{section}
\allowdisplaybreaks

\begin{document}
\graphicspath{{./PIC/}}
\title{Further study on tensor absolute value equations}

\author{Chen Ling \and Weijie Yan \and  Hongjin He \and Liqun Qi}

\institute{C. Ling\and W. Yan\and  H. He \at
Department of Mathematics, School of Science, Hangzhou Dianzi University, Hangzhou, 310018, China.\\
\email{macling@hdu.edu.cn}
\and  W. Yan \at
\email{yanweijie1314@163.com}
\and H. He (\Letter) \at
\email{hehjmath@hdu.edu.cn}
\and L. Qi \at
Department of Applied Mathematics, The Hong Kong Polytechnic University, Hung Hom, Kowloon, Hong Kong. \\
\email{maqilq@polyu.edu.hk}
 }

\date{Received: date / Accepted: date}

\maketitle

\begin{abstract}
In this paper, we consider the {\it tensor absolute value equations} (TAVEs), which is a newly introduced problem in the context of multilinear systems. Although the system of TAVEs is an interesting generalization of matrix {\it absolute value equations} (AVEs), the well-developed theory and algorithms for AVEs are not directly applicable to TAVEs due to the nonlinearity (or multilinearity) of the problem under consideration. Therefore, we first study the solutions existence of some classes of TAVEs with the help of degree theory, in addition to showing,  by fixed point theory, that the system of TAVEs has at least one solution under some checkable conditions. Then, we give a bound of solutions of TAVEs for some special cases. To find a solution to TAVEs, we employ the generalized Newton method and report some preliminary results.
\end{abstract}

\keywords{Tensor absolute value equations \and ${\rm H}^+$-tensor \and P-tensor \and Copositive tensor \and Generalized Newton method.}

\section{Introduction}\label{Introd}
The system of {\it absolute value equations} (AVEs) investigated in literature is given by
\begin{equation}\label{AVEs}
Ax-|x|=b,
\end{equation}
where $A\in \mathbb{R}^{n\times n}$, $b\in \mathbb{R}^n$, and $|x|$ denotes the vector with absolute values of each component of $x$. The importance of AVEs \eqref{AVEs} has been well documented in the monograph  \cite{CPS92} due to its equivalence to the classical {\it linear complementarity problems}. More generally, Rohn \cite{R04} introduced the following problem
\begin{equation}\label{GAVE}
Ax+B|x|=b,
\end{equation}
where $A,B\in \mathbb{R}^{m\times n}$ and $b\in \mathbb{R}^m$. Apparently, \eqref{GAVE} covers \eqref{AVEs} with the setting of $B$ being a negative identity matrix. In what follows, we also call such a general problem \eqref{GAVE} a system of AVEs for simplicity. Since the seminal work \cite{MM06} investigated the existence and nonexistence of solutions to the system of AVEs (\ref{AVEs}) in 2006, the system of AVEs has been studied extensively by many researchers. In the past decade, a series of interesting theoretical results including NP-hardness \cite{M07,MM06}, solvability \cite{MM06,R04}, and equivalent reformulations \cite{M07,MM06,P09} of the system of AVEs have been developed. Also, many efficient algorithms have been designed to solve the system of AVEs, e.g., see \cite{CQZ11,HHZ11,IIA15,M09,MYSC17,ZW09} and references therein.

In the current numerical analysis literature, considerable interests have arisen in extending concepts from linear algebra to the setting of multilinear algebra due to the powerfulness of the multilinear algebra in the real-world applications, e.g. see  \cite{CCLQ18,GLY15,LN15,WQZ09} and the most recent monograph \cite{QCC18}. Therefore, in this paper, we consider the so-named {\it tensor absolute equations} (TAVEs), which refers to the task of finding an $x\in{\mathbb R}^n$ such that
\begin{equation}\label{TAVEs}
{\cal A}x^{p-1}+\mathcal{B}|x|^{q-1} = b,
\end{equation}
where ${\cal A}$ is a $p$-th order $n$-dimensional square tensor, ${\cal B}$ is a $q$-th order $n$-dimensional square tensor, and $b\in \mathbb{R}^n$. In the paper, we are more interested in the case of \eqref{TAVEs} with $p\geq q\geq 2$ due to its real-world applications listed later. Throughout, for given two integers $m$ and $n$, we call $\mathcal{A} = (a_{i_1i_2\ldots i_m} )$, where $a_{i_1i_2\ldots i_m}\in \mathbb{R}$ for $1 \leq i_1,i_2,\ldots,i_m \leq n$, a real $m$-th order $n$-dimensional square tensor. For notational simplicity, we denote the set of all real $m$-th order $n$-dimensional square tensors by $\mathbb{T}_{m,n}$.
Given a tensor $\mathcal{A}=(a_{i_1i_2\ldots i_m})\in \mathbb{T}_{m,n}$ and a vector $x=(x_1,x_2,\ldots,x_n)^\top\in \mathbb{R}^n$, $\mathcal{A}{x}^{m-1}$ is defined as a vector, whose $i$-th component
is given by
\begin{equation}\label{Axm-1}
(\mathcal{A}{x}^{m-1})_i=\sum_{i_2,\ldots,i_m=1}^na_{ii_2\ldots i_m}x_{i_2}\cdots x_{i_m}, \quad i=1,2,\ldots,n.
\end{equation}
Moreover, ${\mathcal A}x^{m-2}$ denotes an $n\times n$ matrix whose $ij$-th component is given by
\begin{equation}\label{Axm-2}
({\mathcal A}x^{m-2})_{ij}:= \sum^n_{i_3, ..., i_m = 1}
a_{ij i_3 ... i_m}x_{i_3}\cdots x_{i_m} ,\quad i,\;j=1,2,\ldots,n.
\end{equation}
Obviously, TAVEs \eqref{TAVEs} becomes the system of AVEs when both tensors ${\mathcal A}$ and ${\mathcal B}$ reduce to matrices, and in particular, TAVEs reduces to the multilinear system (i.e., by taking ${\mathcal B}|x|^{q-1}={\bm 0}$ ) studied in recent work \cite{DW16,HLQZ18,LXX17,XJW18}, which has found many important applications in data mining and numerical partial  differential equations (e.g., see \cite{DW16,LN15}), to name just a few. Most recently, Du et al. \cite{DZCQ18} considered another special case of \eqref{TAVEs} with the setting of ${\mathcal B}$ being a negative $p$-th order $n$-dimensional unit tensor (i.e., ${\mathcal B}|x|^{q-1}$ reduces to $-|x|^{[p-1]}$), which is equivalent to a generalized tensor complementarity problem. Especially, when we consider the case where $p>q=2$ and ${\mathcal B}$ is a negative identity matrix (i.e., ${\mathcal B}|x|^{q-1}=-|x|$), it is clear that the resulting TAVEs \eqref{TAVEs} is equivalent to the following generalized tensor complementarity problem
$$
{\bm 0}\leq (\mathcal{A}x^{p-1}+x-b)\perp (\mathcal{A}x^{p-1}-x-b)\geq {\bm 0}.
$$
Additionally, if we restrict the variable $x$ being nonnegative, the system of TAVEs \eqref{TAVEs} is a fundamental model for characterizing the multilinear pagerank problem (e.g., see \cite{GLY15}). Hence, from the above two motivating examples, we are particularly concerned with the system of TAVEs \eqref{TAVEs} with the case where $p\geq q\geq 2$.

It can be easily seen from the definition of tensor-vector product (see \eqref{Axm-1}) that the system of TAVEs \eqref{TAVEs} is a special system of nonlinear  equations. Hence, all theory and algorithms tailored for the system of AVEs are not easily applicable to TAVEs \eqref{TAVEs} due to the underlying nonlinearity (or multilinearity). Moreover, the potentially nonsmooth term ${\mathcal B}|x|^{q-1}$ in \eqref{TAVEs} would make the theoretical findings, including the existence and boundedness of solutions, different with the cases of smooth nonlinear equations. Therefore, one emergent question is that whether the system of TAVEs \eqref{TAVEs} has solutions or not? If yes, which kind of tensors in \eqref{TAVEs} could ensure the existence of solutions? To answer these questions, most recently, Du et al. \cite{DZCQ18} first studied the special case of \eqref{TAVEs} with a negative unit tensor in the absolute value term (i.e., ${\mathcal B}$ is a negative unit tensor and $p=q$ in \eqref{TAVEs}), where they proved that such a reduced system has a solution for some structured tensors (e.g., ${\mathcal A}$ is a $Z$-tensor). However, the appearance of  a general tensor ${\mathcal B}$ in the absolute value term would completely change the existing results, including the solutions existence and algorithm, tailored for the special case of \eqref{TAVEs} studied in \cite{DZCQ18}.  In this paper, we make a further study on the TAVEs \eqref{TAVEs}. Specifically, we are interested in the general form \eqref{TAVEs}, where we allow the case that the two tensors ${\mathcal A}$ and ${\mathcal B}$ have different order, but with $p\geq q\geq 2$ from an application perspective. First, we prove the nonemptiness and compactness of the solutions set of general TAVEs with the help of degree theory, in addition to showing, by fixed point theory, that the system of TAVEs has at least one solution under some checkable conditions. Then, we derive a bound of solutions of TAVEs with the special case $p=q$. Finally, to find a solution to the general form of TAVEs \eqref{TAVEs} (where we further allow $p<q$), we employ the well-developed generalized Newton method to the problem under consideration. The preliminary computational results show that the simplest generalized Newton method is a highly probabilistic reliable solver for TAVEs.

This paper is organized as follows. In Section \ref{Prelim}, we recall some definitions and basic properties about tensors. In Section \ref{Existence}, we present three sufficient conditions for the solutions existence of the system of TAVEs. Here, the first two theorems on solutions existence are established via the degree-theoretic ideas, and last theorem is proved in the context of fixed point theory. Moreover, in Section \ref{Bounds}, we analyze the bound of solutions for the special case of TAVEs. To find solutions of TAVEs \eqref{TAVEs},  we employ the simplest generalized Newton method and investigate its numerical performance in Section \ref{Alg}. Finally, we complete this paper by drawing some concluding remarks in Section \ref{Conclusion}.

\medskip

\noindent{\bf Notation}. As usual, $\mathbb{R}^n$ denotes the space of $n$-dimensional real column vectors. $\mathbb{R}_+^n=\{x=(x_1,x_2,\ldots,x_n)^\top\in \mathbb{R}^n:x_i\geq 0,~~\forall~i=1,2,\ldots,n\}$. A vector of zeros in a real space of arbitrary dimension will be denoted by ${\bm 0}$. For any $x, y\in \mathbb{R}^n$, the Euclidean inner product is denoted by $x^\top y$, and the Euclidean norm $\|x\|$ is denoted as $\|x\|=\sqrt{x^\top x}$.   For given $\mathcal{A}=(a_{i_1i_2\ldots i_m})\in \mathbb{T}_{m,n}$, if the entries $a_{i_1i_2\ldots i_m}$ are invariant under any permutation of their indices, then $\mathcal{A}$ is called a symmetric tensor.
In particular, for every given index
$i\in [n]:=\{1,2,\ldots,n\}$, if an $(m-1)$-th order $n$-dimensional square tensor $\mathcal{A}_i:=(a_{ii_2\ldots i_m})_{1 \leq i_2,\ldots,i_m \leq n}$ is symmetric, then $\mathcal{A}$ is called a semi-symmetric tensor with respect to the indices $\{i_2,\ldots,i_m\}$.
For given $\mathcal{A}=(a_{i_1i_2\ldots i_m} )\in \mathbb{T}_{m,n}$, denote the $\infty$-norm of $\mathcal{A}$ by $$\|\mathcal{A}\|_{\infty} = \max\limits_{1\leq i\leq n}\displaystyle\sum_{i_2,\ldots,i_m = 1}^n |a_{ii_2\ldots i_m}|,$$
and the (squared) Frobenius norm of ${\mathcal A}$ is defined as the sum of the squares of its elements, i.e.,
$$\|{\mathcal A}\|_{\rm Frob}^2:=\sum_{i_1=1}^n\cdots\sum_{i_m=1}^n a_{i_1i_2\ldots i_m}^2.$$
Denote the unit tensor in $\mathbb{T}_{m,n}$ by $\mathcal{I}=(\delta_{i_1\ldots i_m})$, where $\delta_{i_1\ldots i_m}$ is the Kronecker symbol
$$
\delta_{i_1\ldots i_m}=\left\{
\begin{array}{ll}
1,&\;\;{\rm if~}i_1=\ldots =i_m,\\
0,&\;\;{\rm otherwise}.
\end{array}
\right.
$$
With the notation \eqref{Axm-1}, we define ${\mathcal A}x^m = x^\top ({\mathcal A}x^{m-1})$ for ${\mathcal A}\in {\mathbb T}_{m,n}$ and $x\in {\mathbb R}^n$. Moreover, for a given scalar $s>0$, we denote $x^{[s]}=(x_1^s,x^s_2,\ldots,x_n^s)^\top\in {\mathbb R}^n$.
For a smooth (continuously differentiable) function $F:\mathbb{R}^n\rightarrow \mathbb{R}^n$, we denote the Jacobian of $F$ at $x\in \mathbb{R}^n$ by ${\mathscr D}F(x)$, which is an $n\times n$ matrix.

\section{Preliminaries}\label{Prelim}
In this section, we summarize some definitions and properties on tensors that will be used in the coming analysis.

\begin{definition}\label{ERdef}
 Let ${\mathcal A}\in {\mathbb T}_{p,n}$. We say that ${\mathcal A}$ is an ${\rm H}^+$-tensor, if there exists no $(x,t)\in(\mathbb{R}^n\backslash\{ {\bm 0}\})\times\mathbb{R}_+$ such that
\begin{equation}\label{equation1.2}
({\mathcal A+t\mathcal{I}})x^{p-1}={\bm 0},
\end{equation}
where $\mathcal{I}$ is a unit tensor in $\mathbb{T}_{p,n}$. In particular,  ${\mathcal A}$ is called a ${\rm WH}^+$-tensor if there exists no $(x,t)\in(\mathbb{R}_+^n\backslash\{ {\bm 0}\})\times\mathbb{R}_+$ satisfying \eqref{equation1.2}.
\end{definition}

When the order of ${\mathcal A}$ is $p=2$ (i.e., $\mathcal{A}$ is an $n\times n$ matrix), a ${\rm H}^+$-tensor $\mathcal{A}$ is also called a ${\rm H}^+$-matrix. It is obvious from \eqref{equation1.2} that, ${\mathcal A}$ is an ${\rm H}^+$-matrix if and only if $\mathcal{A}$ has no non-positive eigenvalues.
\begin{definition}[\cite{Qi13}]\label{copositive}
Let $\mathcal{A} \in \mathbb{T}_{p,n}$. We say that $\mathcal{A}$ is a copositive (or strictly copositive) tensor, if $\mathcal{A} x^p\geq 0 ~({\rm or }\; \mathcal{A} x^p> 0)$ for any vector $x \in \mathbb{R}_+^n~({\rm or }\;x\in \mathbb{R}_+^n\backslash\{{\bm 0}\})$.
\end{definition}

\begin{definition}[\cite{SQ15}]\label{def2.3}
Let $\mathcal{A} \in \mathbb{T}_{p,n}$. We say that $\mathcal{A}$ is a P-tensor, if it holds that $\max\limits_{1\leq i\leq n}x_i(\mathcal{A}x^{p-1})_i > 0$ for any vector $x \in \mathbb{R}^n\backslash\{{\bm 0}\}$.
\end{definition}

\begin{proposition}\label{P-WR}
Let $\mathcal{A} \in \mathbb{T}_{p,n}$. If $\mathcal{A}$ is a strictly copositive tensor, then $\mathcal{A}$ is a ${\rm WH}^+$-tensor. If $\mathcal{A}$ is a P-tensor, then $\mathcal{A}$ is an ${\rm H}^+$-tensor.
\end{proposition}
\begin{proof} Let $\mathcal{A}$ be a strictly copositive tensor. Suppose that $\mathcal{A}$ is not a ${\rm WH}^+$-tensor. Then, it follows from Definition \ref{ERdef} that there exists $(\bar x, \bar t)\in (\mathbb{R}_+^n\backslash\{ {\bm 0}\})\times\mathbb{R}_+$ such that (\ref{equation1.2}) holds. Consequently, we know that ${\mathcal A}\bar x^p=-\bar t \sum_{i=1}^n\bar x_i^p\leq 0$, which contradicts to the given condition.

Let $\mathcal{A}$ be a P-tensor. Then $p$  must be even. Suppose that $\mathcal{A}$ is not an ${\rm H}^+$-tensor. Then, it follows from Definition \ref{ERdef} that there exists $(\bar x, \bar t)\in (\mathbb{R}^n\backslash\{ {\bm 0}\})\times\mathbb{R}_+$ such that (\ref{equation1.2}) holds. Therefore, we have
$$
\bar x_i(\mathcal{A}\bar x^{p-1})_i + \bar t\bar x^p_i=0, ~~~\forall ~i=1,2,\ldots,n,
$$
which implies
\begin{equation}\label{TThr}
\max\limits_{1\leq i \leq n}\bar x_i(\mathcal{A}\bar x^{p-1})_i= -\min\limits_{1\leq i \leq n}\bar t \bar x_i^p\leq 0.
\end{equation}
It contradicts to the condition that $\mathcal{A}$ is a P-tensor. The proof is completed.
\qed\end{proof}

We have shown that a strictly copositive tensor must be a ${\rm WH}^+$-tensor, but not conversely. The following example is to show that a ${\rm WH}^+$-tensor is not necessarily a strictly copositive tensor.

\begin{exam}\label{exam11-01}
Consider the case where $p=2$ and  $$
\mathcal{A}=\left[
\begin{array}{cc}
1&4\\
1&-2
\end{array}
\right].
$$
By taking $\bar x=(0,4)^\top\in \mathbb{R}_+^2\backslash \{{\bm 0}\}$, we know $\mathcal{A}\bar x^2=-32<0$, which means that $\mathcal{A}$ is not a strictly copositive tensor. However, we claim that $\mathcal{A}$ is a ${\rm WH}^+$-tensor, i.e., there exists no $(x,t)\in(\mathbb{R}_+^2\backslash\{ {\bm 0}\})\times\mathbb{R}_+$ such that \eqref{equation1.2} holds. Suppose that there exists $(\bar x,\bar t)\in(\mathbb{R}_+^2\backslash\{ {\bm 0}\})\times\mathbb{R}_+$ such that \eqref{equation1.2} holds, i.e.,
\begin{equation}\label{eqta}
\left\{
\begin{array}{l}
\bar x_1+4\bar x_2+\bar t\bar x_1=0\\
\bar x_1-2\bar x_2+\bar t\bar x_2=0.
\end{array}
\right.
\end{equation} Since $\bar x\neq {\bm 0}$, we have
$$
\left|
\begin{array}{cc}
1+\bar t&4\\
1&\bar t-2
\end{array}
\right|=0,
$$
which implies $\bar t^2-\bar t-6=(\bar t-3)(\bar t+2)=0$. Since $\bar t\geq 0$, we obtain $\bar t=3$. Consequently, from \eqref{eqta}, we know that $\bar x_1+\bar x_2=0$, which contradicts to the fact that $\bar x\in \mathbb{R}_+^2\backslash\{{\bm 0}\}$. \end{exam}

It was proved by Qi \cite{Qi05} that  H-eigenvalues exist for an even order real symmetric tensor $\mathcal{A}$, and $\mathcal{A}$ is {\it positive definite} (PD) if and only if all of its
H-eigenvalues are positive, i.e., $\mathcal{A}$ is an ${\rm H}^+$-tensor. Hence, in the symmetric tensor case, the concepts of PD-, P- and ${\rm H}^+$-tensors are identical. We also know that if a tensor $\mathcal{A}\in \mathbb{T}_{p,n}$ is a P-tensor, then $p$ must be even, see \cite{YY14}. So, there does not exist an odd order symmetric ${\rm H}^+$-tensor. However, in the asymmetric case, the conclusion is not true, as showed by the following example, which also shows that  an ${\rm H}^+$-tensor is not necessarily a P-tensor for the asymmetric case.

\begin{exam}\label{exam11-1}
Let $m=3$ and let ${\mathcal A}=(a_{i_1i_2i_3})\in\mathbb{T}_{3,2}$ with $a_{111}=a_{112}=a_{211}=a_{212}=a_{222}=1$,  $a_{221}=-1$ and $a_{121}=a_{122}=0$. Then it is obvious that $\mathcal{A}$ is not a P-tensor, due to the fact that $m$ is an odd number. Moreover,  we claim that $\mathcal{A}$ is an ${\rm H}^+$-tensor, i.e., there are no $(x,t)\in(\mathbb{R}^2\backslash\{{\bm 0}\})\times\mathbb{R}_+$ such that \eqref{equation1.2} holds. In fact, for any $t\in \mathbb{R}$ and $x\in \mathbb{R}^2$ it holds that
 $$
 (\mathcal{A}+t\mathcal{I})x^2=
 \left (
 \begin{array}{c}
 (1+t)x_1^2+x_1x_2 \\
x_1^2+(1+t)x_2^2
\end{array}\right).
 $$
If there exists $(\bar x,\bar t)\in(\mathbb{R}^2\backslash\{{\bm 0}\})\times\mathbb{R}_+$ such that $(\mathcal{A}+\bar t\mathcal{I})\bar x^2={\bm 0}$, then $\bar x_1^2+(1+\bar t)\bar x_2^2=0$. Consequently, since $\bar t\geq 0$, we obtain $\bar x_1=\bar x_2=0$, which is a contradiction. 
\end{exam}
\begin{remark} It is well known that $P$-tensor is a generalization of positive definite tensor, and many structured tensors, such as even order strongly doubly nonnegative tensor \cite{LQ14}, even order strongly completely positive tensor \cite{LQ14,QXX14} and even order Hilbert tensor \cite{SQ14a}, are the special type of positive definite tensors.
Moreover, as shown in Proposition \ref{P-WR} and Example \ref{exam11-1}, the concept of ${\rm H}^+$-tensor is a generalization of $P$-tensor. The set of all $P$-tensors includes many class of important structured tensors as its proper subset, for example, even order nonsingular $H$-tensor with positive diagonal entries \cite{DLQ15}, even order Cauchy tensor \cite{ChbQ15}  with mutually distinct entries of generating vector \cite{DLQ15}, even order strictly diagonally dominated tensor \cite{YY14}, and so on. If an even order $Z$-tensor $\mathcal{A}$ is a $B$-tensor \cite{SQ15b}, then $\mathcal{A}$ is also a $P$-tensor (see \cite[Th. 3.6]{YY14}).
\end{remark}

\begin{definition}\label{def2.2}
Let $\Psi, \Phi:\mathbb{R}^n\rightarrow \mathbb{R}^n $ be two continuous functions. We say that a set of elements $\{x^r\}_{=1}^\infty \subset \mathbb{R}^n$ is an exceptional family of elements
for $\Psi$ with respect to $\Phi$, if the following conditions are satisfied:\\
~~~~ (1) $\|x^r\|\rightarrow \infty$ as $r \rightarrow \infty $,\\
~~~~ (2) for each real number $r>0$, there exists $\mu_r>0$ such that
$$\Psi(x^r)=-\mu_r \Phi(x^r).$$
\end{definition}

\begin{definition}[\cite{Sha13}]\label{def1}
Let $\mathcal{A}$ (and $\mathcal{B}$) be an order $p \geq 2$ (and order $q\geq 1$) dimension $n$ tensor, respectively. Define the product $\mathcal{A} \cdot \mathcal{B}$ to be the following tensor $\mathcal{C}$ of order $(p-1)(q-1)+1$ and dimension $n$:
$$\mathcal{C}_{ij_1\ldots j_{p-1}} = \sum_{i_2,\ldots,i_p=1}^n ~a_{i i_2 \ldots i_p}b_{i_2 j_1} \cdots b_{i_p j_{p-1}}$$
where $i\in [n]$, and $j_1,\ldots,j_{p-1}\in[n]^{q-1}:=\overbrace{[n]\times\cdots\times[n]}^{q-1}$.
\end{definition}

\begin{remark} \label{remarkaB} When $q=1$ (i.e., $\mathcal{B}$ is a vector $x$), it is obvious that $\mathcal{A} \cdot x$ is a vector of dimension $n$, in this case, it holds that $\mathcal{A} \cdot x =\mathcal{A}x^{p-1}$; When $q=2$ (i.e., $\mathcal{B}$ is an $n\times n$ matrix), it is easy to check that $\mathcal{A} \cdot \mathcal{B} $ is a tensor of order $p$; Similarly, when $p=2$ (i.e., $\mathcal{A}$ is an $n\times n$ matrix), we know that $\mathcal{A} \cdot \mathcal{B} $ is a tensor of order $q$. Notice that, in the case when both $\mathcal{A}$ and $\mathcal{B}$ are matrices, or when $\mathcal{A}$ is a matrix and $\mathcal{B}$ is a vector, the tensor product $\mathcal{A}\cdot\mathcal{B}$ coincides with the usual matrix product. So it is a generalization of the matrix product. Here, we refer to \cite{Sha13} for more details.
\end{remark}

\begin{remark}\label{remark THd} Let $\mathcal{A}$ (and $\mathcal{B}, \mathcal{C}$) be an order $(p+1)$ (and order $(q+1)$, order $(m+1)$, respectively) dimension $n$ tensor. Then it holds that $\mathcal{A}\cdot(\mathcal{B}\cdot\mathcal{C})=(\mathcal{A}\cdot\mathcal{B})\cdot\mathcal{C}$. It is easy to check that, when $\mathcal{A}_1$ and $\mathcal{A}_2$ have the same order, we have $(\mathcal{A}_1 + \mathcal{A}_2) \cdot \mathcal{B} = \mathcal{A}_1 \cdot \mathcal{B} + \mathcal{A}_2 \cdot \mathcal{B}$; When $\mathcal{A}$ is a matrix, we have $\mathcal{A}\cdot(\mathcal{B}_1 + \mathcal{B}_2) = \mathcal{A}\cdot\mathcal{B}_1 + \mathcal{A}\cdot\mathcal{B}_2$.
\end{remark}

\begin{definition}[\cite{Sha13}]\label{Inverse} Let $\mathcal{A}\in \mathbb{T}_{p,n}$ and $\mathcal{B}\in \mathbb{T}_{q,n}$. If $\mathcal{A}\cdot \mathcal{B}=\mathcal{I}$, then $\mathcal{A}$ is called an order $p$ left inverse of $\mathcal{B}$, and $\mathcal{B}$ is called an order $q$ right inverse of $\mathcal{A}$.
\end{definition}

From Definition \ref{Inverse}, we know that, for given $\mathcal{A}\in \mathbb{T}_{p,n}$, $\mathcal{A}$ has an order $2$ left inverse if and only if there exists a nonsingular $n\times n$ matrix $Q$ such that $\mathcal{A} = Q \cdot \mathcal{I}$. Moreover, $Q^{-1}$ is the unique order $2$ left inverse of $\mathcal{A}$.

\begin{definition}[\cite{P10}]\label{def4}
Let $\mathcal{A}\in \mathbb{T}_{p,n}$. Then the majorization
matrix $M(\mathcal{A})$ of $\mathcal{A}$ is an $n\times n$ matrix with the entries $M(\mathcal{A})_{ij} = a_{ij\ldots j}$ for all $i,j=1,2,\ldots,n$.
\end{definition}

In \cite{LL16,SY16}, it has been proved that, for given $\mathcal{A}\in \mathbb{T}_{p,n}$, $\mathcal{A}$ has the unique order $2$ left inverse $M(\mathcal{A})^{-1}$, if and only if $M(\mathcal{A})$ is nonsingular and $\mathcal{A}$ is row diagonal (see \cite{SY16}).


\section{Existence of solutions for TAVEs}\label{Existence}
In this section, we focus on studying the existence of solutions of TAVEs \eqref{TAVEs}. The main tools used here are degree-theoretic ideas. We begin this section with recalling some concepts and well-developed necessary results that will play pivot roles in the analysis.

Suppose that $\Omega$ is a bounded open set in $\mathbb{R}^n$, $U: \bar{\Omega}\rightarrow \mathbb{R}^n$ is continuous and $b\not\in U(\partial \Omega)$, where $\bar{\Omega}$ and $\partial \Omega$ denote, respectively, the closure and boundary of $\Omega$. Then the degree of $U$ over $\Omega$ with respect to $b$ is defined, which is an integer and will be denoted by ${\rm deg} (U, \Omega, b)$ (see \cite{FFG95,LN78} for more details on degree theory). If $U(x)=b$ has a unique solution, say, $x^*\in \Omega$, then, ${\rm deg}(U, \Omega, b)$ is constant over all bounded open sets $\Omega^\prime$ containing $x^*$ and contained in $\Omega$. Moreover, we recall the following two fundamental theorems, which can be found in \cite[p. 23]{I06}.

\begin{theorem}[Kronecker's Theorem]\label{th4}
Let $\Omega\subset \mathbb{R}^n$ be a bounded open set, $b\in \mathbb{R}^n$ and $U:\mathbb{R}^n\rightarrow \mathbb{R}^n$ be a continuous function. If ${\rm deg}(U,\Omega,b)$ is defined and non-zero, then the equation $U(x)=b$ has a solution in $\Omega$.
\end{theorem}

\begin{theorem}[Poincar\'{e}-Bohl Theorem]\label{th3}
Let $\Omega\subset \mathbb{R}^n$ be a bounded open set, $b\in \mathbb{R}^n$ and $U, V:\mathbb{R}^n\rightarrow \mathbb{R}^n$ be two continuous functions. If for all $x \in \partial\Omega$ the line segment $[U(x),V(x)]$ does not contain $b$,
then it holds that ${\rm deg}(U,\Omega,b) = {\rm deg}(V,\Omega,b)$.
\end{theorem}

By Theorems \ref{th4} and \ref{th3}, we have the following theorem.
\begin{theorem}\label{th2.6}
Let $G:\mathbb{R}^n\rightarrow \mathbb{R}^n $ be a continuous function. Suppose that $G(x)={\bm 0}$ has only one zero solution and ${\rm deg}(G,B_r,{\bm 0})\neq 0$ for any $r>0$, where $B_r= \{x \in \mathbb{R}^n:\|x\| < r\}$. Then for the continuous function defined by
\begin{equation}\label{F(x)}
F(x)=\mathcal{A}x^{p-1}+\mathcal{B}|x|^{q-1}-b,
\end{equation}
 there exists either a solution to $F(x)={\bm 0}$ or an exceptional family of elements for $F$ with respect to $G$.
\end{theorem}
\begin{proof}
For any real number $r>0$, let us denote the spheres of radius $r$:
$$S_r = \{x \in \mathbb{R}^n:\|x\| = r\}.$$
Obviously, we have $\partial B_r = S_r$. Consider the homotopy between the functions $G$ and $F$, which is defined by:
\begin{equation}
H(x,t) = tG(x) + (1-t)F(x), ~~\forall~(x,t) \in S_r \times [0,1].
\end{equation}
We now apply Theorem \ref{th3} to $H$. There are two cases:\par
~(i) There exists an $r>0$ such that $H(x,t)\neq {\bm 0}$ for any $x\in S_r$ and $t \in [0,1]$. Then by Theorem \ref{th3}, we know that ${\rm deg}(F,B_r,{\bm 0})={\rm deg}(G,B_r,{\bm 0})$. Consequently, it follows from ${\rm deg}(G,B_r,{\bm 0}) \neq 0$ that ${\rm deg}(F,B_r,{\bm 0}) \neq 0$. Moreover, by Theorem \ref{th4}, we know that the ball $B_r$ contains at least one solution to the equation $F(x)={\bm 0}$.\par
~(ii) For each $r>0$, there exists a vector $x^r \in S_r$ (i.e., $\|x^r\|=r$) and a scalar $t_r \in [0,1]$ such that
\begin{equation}
H(x^r,t_r)={\bm 0}.
\end{equation}
If $t_r = 0$, then $x^r$ solves equation $F(x) = {\bm 0}$. If $t_r = 1$, then by the definition of $H(x,t)$ we obtain
$$t_rG(x^r)+ (1-t_r)F(x^r) = G(x^r) = {\bm 0},$$
which implies $x^r={\bm 0}$, since $G(x)={\bm 0}$ has only one zero solution. It contradicts the fact that $\|x^r\| = r > 0$. If $0 < t_r < 1$, then by the definition of $H(x,t)$, we obtain
$$\frac{t_r}{(1-t_r)} G(x^r) + F(x^r)= \frac{1}{1-t_r}(t_r G(x^r) + (1-t_r)F(x^r)) = {\bm 0}. $$
Letting $\mu_r = \frac{t_r}{1-t_r}$, we have $F(x^r)+\mu_r G(x^r)={\bm 0}$. Due to the fact that $\|x^r\| = r$, we know that $\|x^r\|\rightarrow \infty$ as $r \rightarrow \infty $. Thus, from Definition {\ref{def2.2}}, we know that $\{x^r\}$ is an exceptional family of elements
for $F$ with respect to $G$.
\qed\end{proof}

We now state and prove some existence results on solutions of \eqref{TAVEs}. To this end, we first present the following lemma.

\begin{lemma}\label{Lemma01}
Let $m\geq 2$ be a given integer. Then for any vector $x\in \mathbb{R}^n$, it holds that
$$
\|x\|^{m-1}\leq  n^{\frac{m-2}{2}}\|x^{[m-1]}\|.
$$
\end{lemma}

\begin{proof}
The desired result can be proved by the well-known H\"{o}lder inequality.
\qed\end{proof}

We now turn to our first existence theorem, which shows that, in case where $p>q\geq 2$ and $p$ is even,  the system of TAVEs \eqref{TAVEs} has a nonempty and compact solutions set if ${\mathcal A}$ is an ${\rm H}^+$-tensor.

\begin{theorem}\label{Exists}
Let ${\mathcal A}\in {\mathbb T}_{p,n}$ and ${\mathcal B}\in {\mathbb T}_{q,n}$. Suppose that $p$ is an even number satisfying $p>q\geq 2$ and ${\mathcal A}$ is an ${\rm H}^+$-tensor. Then  the solution set of \eqref{TAVEs},
 denoted by ${\rm SOL}(\mathcal{A},\mathcal{B},b)$, is a nonempty compact set for any $b\in \mathbb{R}^n$.
\end{theorem}

\begin{proof}
We first prove that the equation \eqref{TAVEs} always has a solution for any $b\in \mathbb{R}^n$. Letting $G(x)=x^{[p-1]}$, it is easy to see that $G(x)={\bm 0}$ has only one zero solution. Moreover, since ${\bm 0}$ is a critical point of $G$, that is, the determinant of the Jacobian matrix of $G$ at ${\bm 0}$ is zero (i.e., ${\rm det}({\mathscr D}G({\bm 0}))=0$), it follows from Sard's Lemma (see \cite[p. 9]{FFG95}) and Definition 1.9 in \cite[p. 14]{FFG95} that ${\rm deg}(G,B_r,{\bm 0})\neq 0$ for any $r>0$. Suppose that the equation $F(x)={\bm 0}$ does not have solutions, where $F(x)$ is given by \eqref{F(x)}. Then by Theorem \ref{th2.6}, we know that there exists an exceptional family of elements $\{x^r\}_{r>0}$  of $F$ with respect to $G$, i.e., $\{x^r\}_{r>0}$ satisfies $\|x^r\|\rightarrow \infty$ as $r \rightarrow \infty $, and for each real number $r>0$ there exists a $\mu_r>0$ such that
$$\mathcal{A}(x^r)^{p-1}+\mathcal{B}|x^r|^{q-1}-b = -\mu_r(x^r)^{[p-1]},$$
which implies
\begin{equation}\label{TTr}
\mathcal{A}(\bar x^r)^{p-1}+\frac{1}{\|x^r\|^{p-q}}\mathcal{B}|\bar x^r|^{q-1}-\frac{b}{\|x^r\|^{p-1}}= -\mu_r(\bar x^r)^{[p-1]},
\end{equation}
where $\bar x^r=x^r/\|x^r\|$  for any $r$. Since $\|\bar x^r\|=1$ for any $r$, by Lemma \ref{Lemma01}, we know that $n^{-\frac{p-2}{2}}\leq \|(\bar x^r)^{[p-1]}\|$ for any $r$. Consequently, by (\ref{TTr}), it holds that
$$
n^{-\frac{p-2}{2}}\mu_r\leq \left\|\mathcal{A}(\bar x^r)^{p-1}+\frac{1}{\|x^r\|^{p-q}}\mathcal{B}|\bar x^r|^{q-1}-\frac{b}{\|x^r\|^{p-1}}\right\|.
$$
Hence, since $\|x^r\|\rightarrow \infty$ as $r \rightarrow \infty $ and $\|\bar x^r\|=1$ for any $r$, we claim that $\{\mu_r\}_{r>0}$ is bounded. Without loss of generality, we assume that $\bar x^r\rightarrow \bar{x}$ and $\mu_r\rightarrow \bar t$ as $r \rightarrow \infty$. From (\ref{TTr}), by taking $r \rightarrow \infty$, there exists $(\bar{ x},\bar t)\in(\mathbb{R}^n\backslash\{
 {\bm 0}\})\times\mathbb{R}_+$ such that
$$(\mathcal{A} + \bar t\mathcal{I})\bar{x}^{p-1} = {\bm 0},$$
which contradicts to the given condition that $\mathcal{A}$ is an ${\rm H}^+$-tensor.

Hereafter, we prove the compactness of the solution set ${\rm SOL}(\mathcal{A},\mathcal{B},b)$. It is obvious that ${\rm SOL}(\mathcal{A},\mathcal{B},b)$ is closed. We now prove that ${\rm SOL}(\mathcal{A},\mathcal{B},b)$ is bounded for any $b\in \mathbb{R}^n$. Suppose that ${\rm SOL}(\mathcal{A},\mathcal{B},b)$ is unbounded for some $\bar b\in \mathbb{R}^n$, then there exists a sequence $\{x^r\}_{r=1}^\infty$ satisfying $\|x^r\|\rightarrow\infty$ as $r\rightarrow\infty$, such that
$
\mathcal{A}(x^r)^{p-1}+\mathcal{B}|x^r|^{q-1}=\bar b$, which implies
\begin{equation}\label{Bound}
\mathcal{A}(\bar x^r)^{p-1}+\frac{1}{\|x^r\|^{p-q}}\mathcal{B}|\bar x^r|^{q-1}=\frac{\bar b}{\|x^r\|^{p-1}}
\end{equation}
where $\bar x^r=x^r/\|x^r\|$. Without loss of generality, we assume that $\bar x^r\rightarrow \bar{x}$ as $r \rightarrow \infty$. It is clear that $\bar x\neq{\bm 0}$. Consequently, by letting $r\rightarrow \infty$ in (\ref{Bound}), we know $\mathcal{A}\bar x^{p-1}={\bm 0}$, which means that there exists $(\bar x,0)\in (\mathbb{R}^n\backslash\{{\bm 0}\})\times \mathbb{R}_+$ such that \eqref{equation1.2} holds. It is a contradiction. We complete the proof.
\qed\end{proof}

From Theorem \ref{Exists} and Proposition \ref{P-WR}, we immediately obtain the following corollary.
\begin{corollary}\label{corollary1}
Let $\mathcal{A}\in \mathbb{T}_{p,n}$. If $\mathcal{A}$ is a $P$-tensor, then for any $\mathcal{B}\in \mathbb{T}_{q,n}$ with $2\leq q<p$, the system of TAVEs \eqref{TAVEs} has at least one solution.
\end{corollary}

After the discussions on the case $p>q$, ones may be further concerned with the case $p=q$. Below, we give an answer to the solutions existence for such a case $p=q$. We further make the following assumption on the underlying tensors $\mathcal{A}$ and $\mathcal{B}$.

\begin{assum}\label{Assum00}
Let ${\mathcal A},{\mathcal B}\in {\mathbb T}_{p,n}$. Suppose that $(\mathcal{A}+t\mathcal{I})x^{p-1}+\mathcal{B}|x|^{p-1}={\bm 0}$ has no solution for  $(x,t)\in (\mathbb{R}^n\backslash \{{\bm 0}\})\times \mathbb{R}_+$.
\end{assum}

Notice that the set of tensors pair $({\mathcal A},{\mathcal B})$ satisfying Assumption \ref{Assum00} is nonempty, which can be shown by the following example.
\begin{exam}\label{exam11-11}
Let ${\mathcal A}=(a_{i_1i_2i_3i_4})\in\mathbb{T}_{4,2}$ with $a_{2111}=1$, $a_{1222}=-2$ and all other $a_{i_1i_2i_3i_4}=0$. Let ${\mathcal B}=(b_{i_1i_2i_3i_4})\in\mathbb{T}_{4,2}$ with $b_{1111}=-1$, $b_{2222}=1$ and all other $b_{i_1i_2i_3i_4}=0$.

We claim that $(\mathcal{A}+t\mathcal{I})x^{3}+\mathcal{B}|x|^{3}={\bm 0}$ has no solution $(x,t)\in (\mathbb{R}^2\backslash \{{\bm 0}\})\times \mathbb{R}_+$. In fact, suppose there exists $(\bar x,\bar t)\in (\mathbb{R}^2\backslash \{{\bm 0}\})\times \mathbb{R}_+$ such that $(\mathcal{A}+t\mathcal{I})x^{3}+\mathcal{B}|x|^{3}={\bm 0}$, then
\begin{equation}\label{TTde}
\left\{
\begin{array}{l}
(\bar t-\bar \delta_1)\bar x_1^3-2\bar x_2^3=0\\
\bar x_1^3+(\bar t+\bar \delta_2)\bar x_2^3=0,
\end{array}
\right.
\end{equation}
where $\bar \delta_1={\rm sign }(\bar x_1)$ and $\bar \delta_2={\rm sign }(\bar x_2)$ in a componentwise sense, and
$${\rm sign}(\tau)=\left\{\begin{array}{rl}
1, &\quad \tau>0,\\
0, & \quad \tau =0,\\
-1, & \quad \tau<0.
\end{array}\right.$$
Hence, it follows from the definition of ${\rm sign}(\tau)$ that
$$\Delta:=(\bar \delta_2-\bar \delta_1)^2-4(2-\bar \delta_1\bar \delta_2)\leq 2(\bar \delta_1\bar \delta_2-3)<0,$$
which implies the determinant of the matrix of coefficients of \eqref{TTde} must satisfy
$$
\left|
\begin{array}{cc}
\bar t-\bar \delta_1&-2\\
1&\bar t+\bar \delta_2
\end{array}
\right|=\bar t^2+(\bar \delta_2-\bar \delta_1)\bar t-\bar \delta_1\bar \delta_2+2>0.
$$
As a consequence, the solution of \eqref{TTde} is $\bar x_1^3=\bar x_2^3=0$ leading to a contradiction.
\end{exam}

\begin{theorem}\label{Exists2}
Let ${\mathcal A},{\mathcal B}\in {\mathbb T}_{p,n}$ with $p\geq 3$ being an even number. Suppose that $(\mathcal{A},{\mathcal B})$ satisfies Assumption \ref{Assum00}. Then the solution set ${\rm SOL}(\mathcal{A},\mathcal{B},b)$ of \eqref{TAVEs} is a nonempty compact set for any $b\in \mathbb{R}^n$.
\end{theorem}

\begin{proof}
It can be proved by the similar way used in the proof of Theorem \ref{Exists}. Here we skip the proof for brevity. \qed\end{proof}

To close this section, motivated by \cite{LLV18}, we state and prove the following theorem, in which a more checkable condition for the existence of solutions of \eqref{TAVEs} with $p=q$ is presented.
\begin{theorem}\label{Th3}
Let $\mathcal{A},\mathcal{B}\in \mathbb{T}_{p,n}$. Suppose that $p$ is even, and $\mathcal{A}$ has the unique order $2$ left inverse $M(\mathcal{A})^{-1}$. If $||M(\mathcal{A})^{-1}\cdot\mathcal{B}||_\infty < 1$, then \eqref{TAVEs} with $p=q$ has at least one solution for any $b\in \mathbb{R}^n$.
\end{theorem}
\begin{proof}
When $b = {\bm 0}$, it is clear that \eqref{TAVEs} has a zero solution. Now we assume that $b \neq {\bm 0}$. Let
$\mathcal{G} =(g_{i_1\ldots i_p})_{1\leq i_1,\ldots,i_p\leq n}= M(\mathcal{A})^{-1}\cdot\mathcal{B}$ and $h= (h_i)_{1\leq i\leq n} =M(\mathcal{A})^{-1}b$.
By the given condition, we have $||\mathcal{G}||_\infty < 1$.
Taking a parameter $\tau$ with $\tau^{p-1}\geq\frac{||h||_\infty}{1-||\mathcal{G}||_\infty}$, it is obvious that $\tau>0$. Set
$$\Omega = \left\{x=(x_1,x_2,\ldots,x_n)^\top\in \mathbb{R}^n:|x_i|\leq\tau\right\}$$
and
$$f(x) = \left(M(\mathcal{A})^{-1}b-M(\mathcal{A})^{-1}\cdot\mathcal{B}\cdot|x|\right)^{\left[\frac{1}{p-1}\right]}.$$
It is obvious that $\Omega$ is a closed convex set in $\mathbb{R}^n$ and $f$ is continuous. It then follows from the definition of $f$ that
\begin{align*}
|f(x)_i|
& = \left|\left((M(\mathcal{A})^{-1}b)_i-(M(\mathcal{A})^{-1}\cdot\mathcal{B}\cdot|x| )_i\right)^{\frac{1}{p-1}}\right|\\
& =  \left|h_i-\displaystyle\sum_{i_2,\ldots,i_p=1}^n g_{ii_2\ldots i_p}|x_{i_2}|\cdots|x_{i_p}|\right|^{\frac{1}{p-1}}\\
& \leq\left(\|h\|_\infty+\|\mathcal{G}\|_\infty\tau^{p-1}\right)^{\frac{1}{p-1}}\\
&\leq  \tau,
\end{align*}
which shows that $f$ is a map from the set $\Omega$ to itself.  By Brouwer's Fixed Point Theorem (see \cite[p. 125]{Rhe98} or \cite[p. 377]{AH65}), there exists a vector $\bar x\in\Omega$ such that $f(\bar x) = \bar x$, that is,
$$\left(M(\mathcal{A})^{-1}b-M(\mathcal{A})^{-1}\cdot\mathcal{B}\cdot|\bar x|\right)^{\left[\frac{1}{p-1}\right]} = \bar x.$$
Consequently, we have
\begin{equation}\label{equ}
M(\mathcal{A})^{-1}b-M(\mathcal{A})^{-1}\cdot\mathcal{B}\cdot|\bar x|= {\bar x}^{[p-1]}=\mathcal{I}\cdot \bar x,
\end{equation}
where $\mathcal{I}$ is the unit  tensor in $\mathbb{T}_{p,n}$. By Remarks \ref{remarkaB} and \ref{remark THd}, we have that $M(\mathcal{A})\cdot M(\mathcal{A})^{-1}b=b$ and $M(\mathcal{A})\cdot M(\mathcal{A})^{-1}\cdot \mathcal{B}\cdot|\bar x|=\mathcal{B}|\bar x|^{p-1}$. Hence, multiplying both sides of equation (\ref{equ}) by $M(\mathcal{A})$ leads to
$$b-\mathcal{B}\cdot|\bar x|= M(\mathcal{A})\cdot\mathcal{I}\cdot \bar x  = \mathcal{A}\cdot \bar x,$$
which implies $ \mathcal{A}{\bar x}^{p-1}+\mathcal{B}|\bar x|^{p-1} = b$. Therefore, \eqref{TAVEs} has at least one solution.
We complete the proof.\qed\end{proof}

\section{Bound of solutions}\label{Bounds}
In this section, we focus on studying the bound of solutions of \eqref{TAVEs} for the special case $p=q$. We begin with introducing the following concepts on tensors.

\begin{definition}\label{NonsingularDef}
 Let $\mathcal{A}\in \mathbb{T}_{p,n}$, and let $K$ be a given closed convex cone in $\mathbb{R}^n$. We say that $\mathcal{A}$ is $K$-singular, if $\mathcal{A}$ satisfies
 $$\{x\in K\backslash\{\bm 0\}~|~\mathcal{A}x^{p-1}={\bm 0}\}\neq \emptyset.$$
 Otherwise, we say that $\mathcal{A}$ is $K$-nonsingular. In particular, we say that $\mathcal{A}$ is singular, if $\mathcal{A}$ satisfies
 $$\{x\in \mathbb{R}^n\backslash\{\bm 0\}~|~\mathcal{A}x^{p-1}={\bm 0}\}\neq \emptyset.$$
 Otherwise, we say that $\mathcal{A}$ is nonsingular.
\end{definition}

\begin{lemma}\label{aapos}
Let $\mathcal{A}\in \mathbb{T}_{p,n}$. Suppose that $\mathcal{A}$ is nonsingular. Then $\lambda(\mathcal{A})>0$, where $\lambda(\mathcal{A})$ is the optimal value of the following problem
$$
\begin{array}{cl}
{\rm min}& \phi_{\mathcal{A}}(x):=\|\mathcal{A}x^{p-1}\|^2\\
{\rm s.t.}&\|x\|=1.
\end{array}$$
\end{lemma}

\begin{proof}
For any given $\mathcal{A}\in \mathbb{T}_{p,n}$, the objective function $\phi_{\mathcal{A}}(x)$ is continuous on the compact set $\{x\in \mathbb{R}^n~:~\|x\|=1\}$ . It is obvious that the optimal value $\lambda(\mathcal{A})$ exists and is nonnegative at least.

Now, we turn to proving the fact $\lambda(\mathcal{A})>0$. Suppose that $\lambda(\mathcal{A})=0$, then exists an $\bar x\in \mathbb{R}^n$ with $\|\bar x\|=1$ such that $\|\mathcal{A}\bar x^{p-1}\|^2=0$, which implies $\mathcal{A}\bar x^{p-1}={\bm 0}$. It is a contradiction to the condition that $\mathcal{A}$ is nonsingular. Hence, we conclude that $\lambda({\mathcal A})>0$.
\qed\end{proof}

For any given $\mathcal{A}=(a_{i_1i_2\ldots i_p})\in \mathbb{T}_{p,n}$, denote the $2(p-1)$-th order $n$-dimensional square tensor $\mathcal{C}$ by
$$
c_{i_1\ldots i_{p-1}j_1\ldots j_{p-1}}=\sum_{i=1}^na_{ii_1\ldots i_{p-1}}a_{ij_1\ldots j_{p-1}}.
$$
It is clear that, when $p=2$ (i.e, $\mathcal{A}$ is an $n\times n$ matrix $A$), the tensor $\mathcal{C}$ defined above is exactly $A^\top A$. It is easy to see that $\phi_{\mathcal{A}}(x)=\mathcal{C}x^{2(p-1)}$ for any $x\in \mathbb{R}^n$. Moreover, if $\mathcal{A}$ is nonsingular, then $\mathcal{C}$ is a positive definite tensor. Moreover, by Theorem 5 in \cite{Qi05}, we know that $\mathcal{C}$ is positive definite, if and only if all of its $Z$-eigenvalues are positive. Indeed, the optimal value $\lambda(\mathcal{A})$ in Lemma \ref{aapos} is exactly the smallest $Z$-eigenvalue of the tensor $\mathcal{C}$.

\begin{proposition}\label{PropAA}
Let ${\mathcal A}=(a_{i_1\ldots i_{p}})_{1\leq i_1,\ldots,i_{p}\leq n}\in \mathbb{T}_{p,n}$ with $p\geq 3$. For any $x, \tilde{x}\in \mathbb{R}^n$ and $i,j\in [n]$, it holds that
$$
\left|\left({\mathcal A} x^{p-2}-{\mathcal A}\tilde{x}^{p-2}\right)_{ij}\right|\leq \|{\mathcal A}_{ij}\|_{\rm Frob}\|{x}-\tilde{x}\|\sum_{l=0}^{p-3}\|{x}\|^{p-l-3}\|\tilde{x}\|^{l},$$
where ${\mathcal A}_{ij}:=(a_{iji_3\ldots i_{p}})_{1\leq i_3,\ldots,i_{p}\leq n}\in \mathbb{T}_{p-2,n}$.
\end{proposition}

\begin{proof}
For any $x, \tilde{x}\in \mathbb{R}^n$ and every $0\leq l\leq p-2$, denote the $n\times n$ matrix by ${\mathcal A}x^{p-2-l}\tilde{x}^{l}$, whose $ij$-th component is given by
$$
({\mathcal A}x^{p-2-l}\tilde{x}^{l})_{ij}=\sum_{i_3,\ldots,i_p=1}^na_{iji_3\ldots i_{p}}x_{i_3}\cdots x_{i_{p-l}}\tilde{x}_{i_{p-l+1}}\cdots \tilde{x}_{i_p}.
$$
It is easy to see that for every $0\leq l\leq p-3$,
\begin{align*}
&\left|\left({\mathcal A}x^{p-2-l}\tilde{x}^{l}-{\mathcal A}x^{p-3-l}\tilde{x}^{l+1}\right)_{ij}\right|\\
&\leq \sum_{i_3,\ldots,i_p=1}^n\left|a_{iji_3\ldots i_{p}}x_{i_3}\cdots x_{i_{p-l-1}}(x_{i_{p-l}}-\tilde{x}_{i_{p-l}})\tilde{x}_{i_{p-l+1}}\cdots \tilde{x}_{i_p}\right|\\
&\leq\|{\mathcal A}_{ij}\|_{\rm Frob}\|{x}-\tilde{x}\|\|{x}\|^{p-l-3}\|\tilde{x}\|^{l},
\end{align*}
where the second inequality follows from Cauchy-Schwartz inequality. Furthermore, since
$$
\left|\left({\mathcal A}{x}^{p-2}-{\mathcal A}{\tilde{x}}^{p-2}\right)_{ij}\right|\leq \sum_{l=0}^{p-3}\left|({\mathcal A}{ x}^{p-2-l}{\tilde{x}}^{l}-{\mathcal A}{x}^{p-3-l}{\tilde{x}}^{l+1})_{ij}\right|,
$$
it holds that
$$
\left|\left({\mathcal A}{x}^{p-2}-{\mathcal A}{\tilde{x}}^{p-2}\right)_{ij}\right|\leq \|{\mathcal A}_{ij}\|_{\rm Frob}\|{ x}-\tilde{x}\|\sum_{l=0}^{p-3}\|{x}\|^{p-3-l}\|\tilde{x}\|^l.
$$
We obtain the desired result and complete the proof.
\qed\end{proof}

By a similar way used in the proof of Proposition \ref{PropAA}, we can prove the following proposition.
\begin{proposition}\label{PropAB}
Let ${\mathcal A}=(a_{i_1\ldots i_{p}})_{1\leq i_1,\ldots,i_{p}\leq n}\in \mathbb{T}_{p,n}$. For any $x,\tilde{x}\in \mathbb{R}^n$ and $i\in [n]$, it holds that
$$
\left|\left({\mathcal A}x^{p-1}-{\mathcal A}\tilde{x}^{p-1}\right)_i\right|\leq \|{\mathcal A}_i\|_{\rm Frob}\|x-\tilde{x}\|\sum_{l=0}^{p-2}\|x\|^l\|\tilde{x}\|^{p-l-2}.$$
where ${\mathcal A}_{i}:=(a_{ii_2\ldots i_{p}})_{1\leq i_2,\ldots,i_{p}\leq n}\in \mathbb{T}_{p-1,n}$.
\end{proposition}

Applying Proposition \ref{PropAB} to the case where $\tilde{x}={\bm 0}$, we immediately have
\begin{equation}\label{Axnorm}
\|{\mathcal A}x^{p-1}\|\leq \|\mathcal{A}\|_{\rm Frob}\|x\|^{p-1}, \quad \forall x\in \mathbb{R}^n.
\end{equation}


\begin{theorem}\label{boundsolution}
Let $\mathcal{A}\in \mathbb{T}_{p,n}$. Suppose that $\mathcal{A}$ is nonsingular. Then for any $\mathcal{B}\in \mathbb{T}_{p,n}$ with satisfying $\|\mathcal{B}\|_{\rm Frob}<\sqrt{\lambda(\mathcal{A})}$, it holds that
$$
\|x\|\leq \frac{(\sigma+\|b\|)^{\frac{1}{p-1}}}{\lambda(\mathcal{A})^{\frac{1}{2(p-1)}}-\|\mathcal{B}\|_{\rm Frob}^{\frac{1}{p-1}}}
, \quad \forall x\in L_\sigma,$$
where $L_\sigma:=\{x\in \mathbb{R}^n~:~\|F(x)\|\leq \sigma\}$ and $F(x)$ is defined by \eqref{F(x)} with $p=q$.
\end{theorem}

\begin{proof}
For any $x\in L_\sigma$, it holds that
\begin{align}\label{mmnorm}
\|F(x)\|&=\|\mathcal{A}x^{p-1}+\mathcal{B}x^{p-1}-b\| \nonumber\\
&\geq \|\mathcal{A}x^{p-1}\|-\|\mathcal{B}x^{p-1}\|-\|b\| \nonumber \\
&\geq  \sqrt{\lambda(\mathcal{A})}\|x\|^{p-1}-\|\mathcal{B}\|_{\rm Frob}\|x\|^{p-1}-\|b\|,
\end{align}
where the last inequality comes from Lemma \ref{aapos} and inequality (\ref{Axnorm}). By (\ref{mmnorm}), we obtain
 $$
 \sqrt{\lambda(\mathcal{A})}\|x\|^{p-1}\leq \|F(x)\|+\|b\|+\|\mathcal{B}\|_{\rm Frob}\|x\|^{p-1},
 $$
 which, together with an application of $({\bm u}+{\bm v})^{\frac{1}{p-1}}\leq {\bm u}^{\frac{1}{p-1}}+{\bm v}^{\frac{1}{p-1}} $ for any ${\bm u}, {\bm v}\in \mathbb{R}_+$, implies
\begin{equation}\label{jjmmnorm}
\lambda(\mathcal{A})^{\frac{1}{2(p-1)}}\|x\|\leq (\|F(x)\|+\|b\|)^{\frac{1}{p-1}}+\|\mathcal{B}\|_{\rm Frob}^{\frac{1}{p-1}}\|x\|,
\end{equation}
Hence, by the given condition that $\|F(x)\|\leq \sigma$, we obtain the desired result and complete the proof.
\qed\end{proof}

\begin{remark}\label{solutionRem} Let $\mathcal{A}, \mathcal{B}\in \mathbb{T}_{p,n}$. For any solution $x$ of the special case of \eqref{TAVEs} with $p=q$, i.e., $\mathcal{A}x^{p-1}+\mathcal{B}|x|^{p-1}=b$, it follows from \eqref{Axnorm} that
$$\|b\|\leq \|\mathcal{A}x^{p-1}\|+\|\mathcal{B}|x|^{p-1}\|\leq \|\mathcal{A}\|_{\rm Frob}\|x\|^{p-1}+\|\mathcal{B}\|_{\rm Frob}\||x|\|^{p-1},$$
which implies
$$
\|x\|\geq \left\{\frac{\|b\|}{\|\mathcal{A}\|_{\rm Frob}+\|\mathcal{B}\|_{\rm Frob}}\right\}^\frac{1}{p-1}.
$$
If $\mathcal{A}$ is nonsingular, and $\|\mathcal{B}\|_{\rm Frob}<\sqrt{\lambda(\mathcal{A})}$, then from Theorem \ref{boundsolution}, it holds that
$$
\|x\|\leq \frac{\|b\|^{\frac{1}{p-1}}}{\lambda(\mathcal{A})^{\frac{1}{2(p-1)}}-\|\mathcal{B}\|_{\rm Frob}^{\frac{1}{p-1}}}
$$
for any solution $x$ of \eqref{TAVEs}.
\end{remark}

\section{Algorithm and numerical results}\label{Alg}
In this section, we will employ the well-developed generalized Newton method to find a numerical solution of the system of TAVEs \eqref{TAVEs}. So, we first present the details of the generalized Newton method for solving TAVEs. Then, to show the numerical performance, we report some results by testing synthetic examples with random data.

\subsection{Algorithm}\label{Algorithm}
At the beginning of this section, we first list two lemmas, which open a door of applying the generalized Newton method to TAVEs \eqref{TAVEs}. Here, we refer the reader to \cite{M09} for the proofs.

\begin{lemma}\label{singularv} The singular values of the matrix $A\in \mathbb{R}^{n\times n}$ exceed $1$ if and only if the
minimum eigenvalue of $A^\top A$ exceeds $1$.
\end{lemma}

\begin{lemma}\label{singularA} If the singular values of $A\in \mathbb{R}^{n\times n}$ exceed $1$ then $A+D$ is invertible for
any diagonal matrix $D$ whose diagonal elements equal $\pm 1$ or $0$.
\end{lemma}

\begin{remark}\label{remarkAB}
Note that the definition \eqref{Axm-2} corresponds to a matrix. Then, we can apply Lemmas \ref{singularv} and \ref{singularA} to the problem under consideration. Specifically, if $\mathcal{B}|x|^{q-2}$ is invertible and the singular values of $(\mathcal{B}|x|^{q-2})^{-1}\mathcal{A}x^{p-2}$ exceed $1$, then by Lemma \ref{singularA}, we immediately know that  $\mathcal{A}x^{p-2}+\mathcal{B}|x|^{q-2}D(x)$ is invertible, where $D(x)={\rm diag}({\rm sign}(x))$ is a diagonal matrix whose diagonal elements are $\pm 1$ or $0$. It is a good news for the employment of the generalized Newton method for TAVEs.
\end{remark}

Below, we first use an example to show the conclusion in Remark \ref{remarkAB} that $\mathcal{A}x^{p-2}+\mathcal{B}|x|^{q-2}D(x)$ is invertible under some conditions.

\begin{exam}\label{exam11-11}
Let $\mathcal{B}=(b_{ijk})\in \mathbb{T}_{3,2}$ with $b_{111}=b_{222}=0$ and $b_{112}=b_{121}=b_{211}=b_{122}=b_{212}=b_{221}=1$. Then for any $x\in \mathbb{R}^2\backslash \{0\}$, we have
$$
\mathcal{B}x=\left[
\begin{array}{cc}
x_2&x_1+x_2\\
x_1+x_2&x_1
\end{array}
\right].
$$
Consequently, we have ${\rm det}(\mathcal{B}x)=-(x_1^2+x_1x_2+x_2^2)<0$ for any $x\in \mathbb{R}^2\backslash \{0\}$, which means that $\mathcal{B}x$ is invertible. Let $\mathcal{A}=(a_{ijk})\in \mathbb{T}_{3,2}$ with $a_{111}=a_{222}=0$ and $a_{112}=a_{121}=a_{211}=a_{122}=a_{212}=a_{221}=2$. Then it is easy to see that the singular values of $(\mathcal{B}x)^{-1}\mathcal{A}x$ exceed $1$, and $\mathcal{A}x+\mathcal{B}xD(x)$ is invertible.
\end{exam}\label{exam11-11}

Now, we present the generalized Newton method for TAVEs. Recalling the notation \eqref{F(x)}, we consider the case where ${\mathcal A}\in {\mathbb T}_{p,n}$ and ${\mathcal B}\in {\mathbb T}_{q,n}$ are semi-symmetric tensors. Denote
\begin{equation}\label{Vk}
V(x) = (p-1) {\mathcal A}x^{p-2} + (q-1){\mathcal B}|x|^{q-2}D(x),
\end{equation}
where $D(x)$ is given in Remark \ref{remarkAB}. Then, the matrix $V(x)$ defined by \eqref{Vk} can be viewed as a generalized Jacobian matrix of $F$ at $x$.
Then, it follows from \cite{QS99} that, for a given $x_k$, the generalized Newton method for \eqref{TAVEs} reads as follows:
\begin{equation}\label{GN}
x_{k+1} = x_k  - V(x_k)^{-1} F(x_k),
\end{equation}
where $V(x_k)$ stands for the generalized Jacobian matrix at $x_k$. By utilizing the notation $F(x_k)$ and $V(x_k)$ and the tensor-vector product \eqref{Axm-1},  the iterative scheme \eqref{GN} can be rewritten as
\begin{equation}\label{GNewton}
x_{k+1}=V(x_k)^{-1}\left[(p-2){\mathcal A}x_k^{p-1}+(q-2){\mathcal B}|x_k|^{q-1}+b\right].
\end{equation}

\begin{remark}
When we consider the case $p=q$ in TAVEs \eqref{TAVEs}, it can be easily seen that the iterative scheme \eqref{GNewton} immediately reduces to
\begin{align*}
x_{k+1}&=\frac{p-2}{p-1}\left[{\mathcal A}x_k^{p-2}+{\mathcal B}|x_k|^{p-2}D(x_k)\right]^{-1}\left[\left({\mathcal A}x_k^{p-2}+{\mathcal B}|x_k|^{p-2}D(x_k)\right)x_k+\frac{b}{p-2}\right] \\
&=\frac{p-2}{p-1}x^k + \frac{1}{p-1}\left[{\mathcal A}x_k^{p-2}+{\mathcal B}|x_k|^{p-2}D(x_k)\right]^{-1}b,
\end{align*}
where the first equality uses the fact that $D(x_k)x_k = |x_k|$. In particular, if we consider the special case without the absolute value term (i.e., $ {\mathcal B}|x|^{p-1}={\bf 0}$) in \eqref{TAVEs}, the above iterative scheme immediately recovers the Newton method introduced in \cite{LXX17}.
\end{remark}

\subsection{Numerical results}
We have proposed a generalized Newton method \eqref{GNewton} for the system of TAVEs \eqref{TAVEs} in Section \ref{Algorithm}. It is not difficult to see that the generalized Newton method enjoys a simple iterative scheme. In this subsection, we will show through experimentation with synthetic data that such a simple method is a highly probabilistic reliable TAVEs solver for the problem under consideration.

We write the code of the generalized Newton method in {\sc Matlab} 2014a and conduct the experiments on a DELL workstation computer equipped with Intel(R) Xeon(R) CPU E5-2680 v3 @2.5GHz and 128G RAM running on Windows 7 Home Premium operating system. Here, we employ the publicly shared {\sc Matlab} Tensor Toolbox \cite{TensorT} to compute tensor-vector products and symmetrization of tensors.

From an application perspective, we only established the solutions existence theorems for the case of TAVEs \eqref{TAVEs} with $p\geq q\geq 2$. However, the proposed generalized Newton method does not depend on the relation between $p$ and $q$, i.e., the algorithm is applicable to the cases $p\geq q$ and $p\leq q$. Therefore, we consider two cases of TAVEs with $p=q$ and $p\neq q$ (i.e., $p<q$ and $p>q$) in our experiments. Moreover, we investigate four scenarios on tensors ${\mathcal A}$ and ${\mathcal B}$: (i) both ${\mathcal A}$ and ${\mathcal B}$ are ${\mathcal M}$-tensors; (ii) ${\mathcal A}$ is an ${\mathcal M}$-tensor, ${\mathcal B}$  is a general random tensor; (iii) ${\mathcal A}$ is a general random tensor, ${\mathcal B}$ is an ${\mathcal M}$-tensor; (iv) both ${\mathcal A}$ and ${\mathcal B}$ are general random tensors. Here, to generate an ${\mathcal M}$-tensor ${\mathcal A}$ or ${\mathcal B}$, we follow the way used in \cite{DW16}. That is, we first generate a random tensor ${\mathcal C}=(c_{i_1 i_2\ldots i_m})$ and set
$$\zeta_{\mathcal C}=(1+\epsilon)\cdot \max_{1\leq i\leq n}\left(\sum_{i_2,\cdots,i_m=1}^nc_{ii_2\ldots i_m}\right),\quad \epsilon>0.$$
Then, we take ${\mathcal A}$ or ${\mathcal B}$ as $\zeta_{\mathcal C}{\mathcal I}-{\mathcal C}$. More concretely, for the above four scenarios: (i) We first generate ${\mathcal C}_1$ and ${\mathcal C}_2$ randomly so that all entries are uniformly distributed in $(0,1)$ and $(-1,1)$, respectively. Then, we take ${\mathcal A} = \zeta_{\mathcal C_1}{\mathcal I}-{\mathcal C}_1$ and ${\mathcal B} = \zeta_{\mathcal C_2}{\mathcal I}-{\mathcal C}_2$; (ii) Generate ${\mathcal C}$ whose entries are uniformly distributed in $(0,2)$ and take ${\mathcal A} = \zeta_{\mathcal C}{\mathcal I}-{\mathcal C}$. Tensor ${\mathcal B}$ is a general random one whose components are uniformly distributed in $(-1,0)$; (iii) ${\mathcal A}$ is a general random tensor whose components are uniformly distributed in $(-1,0)$. For tensor ${\mathcal B}$, we generate ${\mathcal C}$ such that all entries are uniformly distributed in $(-0.5,0.5)$ and take ${\mathcal B} = \zeta_{\mathcal C}{\mathcal I}-{\mathcal C}$; (iv) Both ${\mathcal A}$ and ${\mathcal B}$ are general tensors, whose entries are uniformly distributed in $(0,1)$ and $(-4,1)$, respectively. Throughout, we take $\epsilon=0.1$ for all ${\mathcal M}$-tensors, and all tensors ${\mathcal A}$ and ${\mathcal B}$ are symmetrized by the {\sc Matlab} tensor toolbox \cite{TensorT}. To keep the fact that each randomly generated problem has at least one solution, we construct $b$ by setting $b={\mathcal A}x_*^{p-1}+{\mathcal B}|x_*|^{q-1}$, where $x_*$ is a pregenerated vector whose entries are uniformly distributed in $(-1,1)$. Moreover, we always take $x_0=(1,1,\cdots,1)^\top$ as our initial point for the proposed method.

To investigate the numerical performance of the generalized Newton method, we report the number of iterations (Iter.), the computing time in seconds (Time), the absolute error (Err) at point $x_k$, which is defined by
$${\rm Err}:=\|{\mathcal A}x_k^{p-1} + {\mathcal B}|x_k|^{q-1}-b\|\leq {\rm Tol}.$$
Throughout, we set ${\rm Tol}=10^{-5}$. Since all the data is generated randomly, we test $100$ groups of random data for each scenario and report the minimum and maximum iterations ($k_{\min}$ and $k_{\max}$), the minimum and maximum computing time ($t_{\min}$ and $t_{\max}$), respectively. In practice, notice that we completely do not know the true solutions of the system of TAVEs \eqref{TAVEs}. Hence, we can not guarantee that the generalized Newton method starting with the constant initial point $x_0$ (which might be far away from the true solutions) is always convergent (or successful) for the random data. Accordingly, we report the {\it success rate} (SR) of $100$ random problems in the sense that the generalized Newton method can achieve the preset `Tol' in 2000 iterations.

\setlength\rotFPtop{0pt plus 1fil}
\begin{sidewaystable}
\begin{center}
\caption{Computational results for the cases $p\equiv q=m$ with (i) $({\mathcal A},{\mathcal B})$ are ${\mathcal M}$-tensors, and (ii) ${\mathcal A}$ is an ${\mathcal M}$-tensor and ${\mathcal B}$ is a general tensor.}\vskip 0.2mm
\label{table1}
\def\temptablewidth{1\textwidth}
\begin{tabular*}{\temptablewidth}{@{\extracolsep{\fill}}llll}\toprule
& (i) $({\mathcal A},{\mathcal B})$ are ${\mathcal M}$-tensors && (ii) ${\mathcal A}$ is an ${\mathcal M}$-tensor, ${\mathcal B}$ is a general tensor \\\cline{2-2} \cline{4-4}
$(m,n)$ & Iter. ($k_{\min}$ / $k_{\max}$) / Time ($t_{\min}$ / $t_{\max}$) / Err / SR && Iter. ($k_{\min}$ / $k_{\max}$) / Time ($t_{\min}$ / $t_{\max}$) / Err / SR \\ \midrule
$( 3, 5)$ & 5.16 (  3 / 103) / 0.02  (0.00 / 0.44)  / 1.39$\times 10^{-6}$ / 1.00 && 39.39 (  4 / 939) / 0.17  (0.02 / 3.96)  / 1.53$\times 10^{-6}$ / 1.00\\
$( 3,10)$ & 3.57 (  3 /   5) / 0.02  (0.00 / 0.03)  / 1.56$\times 10^{-6}$ / 1.00 && 70.24 (  6 / 1401) / 0.30  (0.03 / 6.01)  / 1.40$\times 10^{-6}$ / 1.00\\
 $( 3,20)$ & 3.91 (  3 /   5) / 0.02  (0.02 / 0.03)  / 8.89$\times 10^{-7}$ / 1.00 && 108.55 (  8 / 1825) / 0.47  (0.03 / 7.91)  / 1.30$\times 10^{-6}$ / 0.98\\ \midrule
$( 4, 5)$ & 21.83 (  4 / 652) / 0.10  (0.02 / 2.89)  / 1.31$\times 10^{-6}$ / 0.84 && 6.49 (  3 /  36) / 0.03  (0.02 / 0.17)  / 1.11$\times 10^{-6}$ / 0.99 \\
$( 4,10)$ & 11.67 (  4 / 203) / 0.05  (0.02 / 0.89)  / 1.04$\times 10^{-6}$ / 0.94 && 10.03 (  4 / 207) / 0.05  (0.02 / 0.92)  / 8.32$\times 10^{-7}$ / 1.00\\
 $( 4,20)$ & 5.08 (  4 /  18) / 0.02  (0.02 / 0.09)  / 1.13$\times 10^{-6}$ / 1.00 && 18.84 (  4 / 688) / 0.09  (0.02 / 3.21)  / 1.13$\times 10^{-6}$ / 1.00 \\ \midrule
$( 5, 5)$ & 11.46 (  3 /  57) / 0.05  (0.02 / 0.25)  / 1.43$\times 10^{-6}$ / 0.96 && 76.48 (  6 / 710) / 0.35  (0.03 / 3.23)  / 1.55$\times 10^{-6}$ / 1.00\\
 $( 5,10)$ & 10.87 (  3 / 217) / 0.05  (0.00 / 1.11)  / 1.09$\times 10^{-6}$ / 1.00 && 86.66 (  9 / 1925) / 0.41  (0.05 / 9.09)  / 1.83$\times 10^{-6}$ / 0.96 \\
 $( 5,20)$ & 14.10 (  4 / 597) / 0.27  (0.06 / 11.62)  / 1.29$\times 10^{-6}$ / 1.00 && 84.70 ( 10 / 1355) / 1.64  (0.19 / 26.29)  / 1.57$\times 10^{-6}$ / 0.99\\ \midrule
$( 6, 5)$ & 5.46 (  3 /   9) / 0.03  (0.02 / 0.05)  / 1.09$\times 10^{-6}$ / 1.00 && 15.65 (  3 / 146) / 0.08  (0.02 / 0.69)  / 1.25$\times 10^{-6}$ / 0.97 \\
$( 6,10)$ & 5.72 (  3 /  95) / 0.04  (0.02 / 0.67)  / 1.10$\times 10^{-6}$ / 1.00 && 15.66 (  4 / 165) / 0.11  (0.02 / 1.20)  / 1.45$\times 10^{-6}$ / 0.99\\
$( 6,15)$ & 4.97 (  3 /   6) / 0.30  (0.17 / 0.37)  / 1.28$\times 10^{-6}$ / 1.00 && 72.15 (  4 / 1811) / 4.60  (0.23 / 115.89)  / 1.72$\times 10^{-6}$ / 0.98 \\ \bottomrule
\end{tabular*}
\end{center}
\end{sidewaystable}

\setlength\rotFPtop{0pt plus 1fil}
\begin{sidewaystable}
\begin{center}
\caption{Computational results for the cases $p\equiv q=m$ with (iii) ${\mathcal A}$ is a general tensor and ${\mathcal B}$ is an ${\mathcal M}$-tensor, and (iv) $({\mathcal A},{\mathcal B})$ are general tensors}\vskip 0.2mm
\label{table2}
\def\temptablewidth{1\textwidth}
\begin{tabular*}{\temptablewidth}{@{\extracolsep{\fill}}llll}\toprule
& (iii) ${\mathcal A}$ is a general tensor, ${\mathcal B}$ is an ${\mathcal M}$-tensor  && (iv) $({\mathcal A},{\mathcal B})$ are general tensors\\\cline{2-2} \cline{4-4}
$(m,n)$ & Iter. ($k_{\min}$ / $k_{\max}$) / Time ($t_{\min}$ / $t_{\max}$) / Err / SR && Iter. ($k_{\min}$ / $k_{\max}$) / Time ($t_{\min}$ / $t_{\max}$) / Err / SR \\ \midrule
$( 3, 5)$ &  20.86 (  4 / 285) / 0.09  (0.02 / 1.20)  / 1.24$\times 10^{-6}$ / 0.97 && 18.28 (  4 / 293) / 0.08  (0.02 / 1.23)  / 1.31$\times 10^{-6}$ / 0.98 \\
$( 3,10)$ &  46.35 (  6 / 1524) / 0.20  (0.02 / 6.52)  / 8.80$\times 10^{-7}$ / 0.94 && 56.03 (  6 / 844) / 0.24  (0.03 / 3.60)  / 1.39$\times 10^{-6}$ / 0.97 \\
$( 3,20)$ & 88.18 (  8 / 1229) / 0.38  (0.03 / 5.34)  / 9.31$\times 10^{-7}$ / 0.99 && 135.61 (  8 / 551) / 0.59  (0.03 / 2.39)  / 1.45$\times 10^{-6}$ / 1.00\\ \midrule
$( 4, 5)$ & 47.45 (  5 / 965) / 0.21  (0.02 / 4.27)  / 9.84$\times 10^{-7}$ / 0.97 && 28.34 (  5 / 191) / 0.13  (0.02 / 0.84)  / 1.75$\times 10^{-6}$ / 1.00\\
$( 4,10)$ & 62.09 (  8 / 397) / 0.28  (0.03 / 1.76)  / 1.46$\times 10^{-6}$ / 0.92 && 58.61 (  7 / 252) / 0.26  (0.03 / 1.11)  / 1.47$\times 10^{-6}$ / 1.00\\
$( 4,20)$ & 123.13 ( 10 / 615) / 0.58  (0.05 / 2.89)  / 1.24$\times 10^{-6}$ / 0.97 && 332.71 ( 15 / 1933) / 1.55  (0.08 / 9.06)  / 1.10$\times 10^{-6}$ / 0.98 \\ \midrule
$( 5, 5)$ & 18.95 (  5 /  72) / 0.09  (0.02 / 0.34)  / 1.17$\times 10^{-6}$ / 0.95 && 30.49 (  7 / 145) / 0.14  (0.03 / 0.67)  / 1.29$\times 10^{-6}$ / 0.99\\
$( 5,10)$ & 54.46 ( 10 / 439) / 0.26  (0.05 / 2.07)  / 1.63$\times 10^{-6}$ / 0.97 && 81.99 (  8 / 468) / 0.39  (0.05 / 2.20)  / 1.37$\times 10^{-6}$ / 0.99 \\
$( 5,20)$ & 139.13 ( 15 / 1250) / 2.69  (0.28 / 24.21)  / 1.35$\times 10^{-6}$ / 0.95 && 499.55 ( 24 / 1908) / 9.69  (0.47 / 37.08)  / 1.54$\times 10^{-6}$ / 0.93\\ \midrule
$( 6, 5)$ & 42.87 (  7 / 495) / 0.21  (0.03 / 2.34)  / 1.35$\times 10^{-6}$ / 0.97 && 45.54 (  9 / 326) / 0.22  (0.03 / 1.54)  / 1.09$\times 10^{-6}$ / 1.00 \\
$( 6,10)$ & 91.54 ( 10 / 742) / 0.66  (0.06 / 5.43)  / 1.36$\times 10^{-6}$ / 0.92 && 104.55 ( 16 / 697) / 0.76  (0.11 / 4.96)  / 1.01$\times 10^{-6}$ / 1.00 \\
$( 6,15)$ & 176.05 ( 12 / 1188) / 11.24  (0.75 / 76.16)  / 1.81$\times 10^{-6}$ / 0.88 && 274.35 ( 17 / 1421) / 17.53  (1.06 / 90.87)  / 1.27$\times 10^{-6}$ / 0.99 \\
\bottomrule
\end{tabular*}
\end{center}
\end{sidewaystable}

In Tables \ref{table1} and \ref{table2}, we report the results for the case $p\equiv q=m$ with four scenarios on tensors. From the data, we can see that most of the random problems (even with general tensors) can be solved successfully in the preset maximum iteration. When both ${\mathcal A}$ and ${\mathcal B}$ are ${\mathcal M}$-tensors, the generalized Newton method performs best in terms of taking the least average iterations and the highest success rate. For the other three scenarios on tensors, it seems that the number of iterations is proportional to the dimensionality $n$. However, the proposed method is still highly probabilistic reliable to the problem under test.

\setlength\rotFPtop{0pt plus 1fil}
\begin{sidewaystable}
\begin{center}
\caption{Computational results for the cases $p\neq q$ with (i) $({\mathcal A},{\mathcal B})$ are ${\mathcal M}$-tensors, and (ii) ${\mathcal A}$ is an ${\mathcal M}$-tensor and ${\mathcal B}$ is a general tensor.}\vskip 0.2mm
\label{table3}
\small
\def\temptablewidth{1\textwidth}
\begin{tabular*}{\temptablewidth}{@{\extracolsep{\fill}}llll}\toprule
& (i) $({\mathcal A},{\mathcal B})$ are ${\mathcal M}$-tensors && (ii) ${\mathcal A}$ is an ${\mathcal M}$-tensor and ${\mathcal B}$ is a general tensor \\\cline{2-2} \cline{4-4}
$(p,q,n)$ & Iter. ($k_{\min}$ / $k_{\max}$) / Time ($t_{\min}$ / $t_{\max}$) / Err / SR && Iter. ($k_{\min}$ / $k_{\max}$) / Time ($t_{\min}$ / $t_{\max}$) / Err / SR \\ \midrule
$( 4,3, 5)$ & 35.47 (  4 / 150) / 0.15  (0.02 / 0.66)  / 1.10$\times 10^{-6}$ / 0.95 && 23.04 (  3 / 140) / 0.10  (0.00 / 0.62)  / 1.66$\times 10^{-6}$ / 0.99 \\
 $( 4,3,10)$ & 93.28 (  4 / 525) / 0.41  (0.02 / 2.29)  / 2.30$\times 10^{-6}$ / 0.99 && 40.77 (  4 / 134) / 0.18  (0.02 / 0.59)  / 2.05$\times 10^{-6}$ / 1.00 \\
 $( 4,3,20)$ & 121.19 (  4 / 1279) / 0.54  (0.02 / 5.68)  / 2.84$\times 10^{-6}$ / 0.98 && 54.37 (  4 / 173) / 0.24  (0.02 / 0.78)  / 2.70$\times 10^{-6}$ / 0.98  \\ \hline
$( 5,4, 5)$ & 4.54 (  3 /  35) / 0.02  (0.02 / 0.17)  / 1.03$\times 10^{-6}$ / 0.97 && 5.00 (  3 /   7) / 0.02  (0.02 / 0.03)  / 1.12$\times 10^{-6}$ / 1.00 \\
 $( 5,4,10)$ & 4.31 (  3 /   5) / 0.02  (0.00 / 0.03)  / 1.26$\times 10^{-6}$ / 1.00 && 4.89 (  3 /   6) / 0.02  (0.00 / 0.03)  / 9.98$\times 10^{-7}$ / 1.00 \\
 $( 5,4,20)$ & 4.50 (  3 /   5) / 0.05  (0.03 / 0.06)  / 1.09$\times 10^{-6}$ / 1.00 && 4.67 (  3 /   6) / 0.05  (0.03 / 0.06)  / 1.14$\times 10^{-6}$ / 1.00 \\ \hline
$( 6,4, 5)$ & 46.38 (  3 / 551) / 0.22  (0.02 / 2.53)  / 1.52$\times 10^{-6}$ / 0.98 && 48.04 (  4 / 243) / 0.23  (0.02 / 1.12)  / 1.88$\times 10^{-6}$ / 0.98  \\
 $( 6,4,10)$ & 80.79 (  3 / 312) / 0.47  (0.02 / 1.76)  / 2.57$\times 10^{-6}$ / 0.99 && 88.90 (  3 / 473) / 0.51  (0.02 / 2.73)  / 2.53$\times 10^{-6}$ / 0.99 \\
 $( 6,4,15)$ & 129.11 (  4 / 685) / 4.45  (0.12 / 23.71)  / 2.93$\times 10^{-6}$ / 0.98 && 111.10 (  4 / 693) / 3.83  (0.12 / 23.90)  / 3.77$\times 10^{-6}$ / 1.00  \\ \hline
$( 6,5, 5)$ & 67.62 (  3 / 823) / 0.32  (0.02 / 3.82)  / 9.76$\times 10^{-7}$ / 0.92 && 42.61 (  3 / 329) / 0.20  (0.02 / 1.53)  / 1.69$\times 10^{-6}$ / 0.98  \\
 $( 6,5,10)$ & 178.96 (  4 / 814) / 1.07  (0.02 / 4.76)  / 2.11$\times 10^{-6}$ / 0.96 && 77.94 (  3 / 336) / 0.46  (0.02 / 2.03)  / 2.23$\times 10^{-6}$ / 0.98 \\
 $( 6,5,15)$ & 289.66 (  4 / 1309) / 10.42  (0.12 / 47.28)  / 2.14$\times 10^{-6}$ / 0.97 && 125.82 (  3 / 471) / 4.52  (0.09 / 16.85)  / 2.54$\times 10^{-6}$ / 0.98 \\ \midrule
$( 3,4, 5)$ & 10.33 (  3 /  53) / 0.05  (0.02 / 0.23)  / 1.42$\times 10^{-6}$ / 0.96 && 55.83 (  4 / 1063) / 0.25  (0.02 / 4.68)  / 1.08$\times 10^{-6}$ / 0.84 \\
 $( 3,4,10)$ & 5.14 (  4 /  21) / 0.02  (0.02 / 0.09)  / 7.96$\times 10^{-7}$ / 1.00 && 643.21 (  5 / 1739) / 2.83  (0.02 / 7.68)  / 1.66$\times 10^{-6}$ / 0.29  \\
 $( 3,4,20)$ & 4.12 (  4 /   5) / 0.02  (0.02 / 0.03)  / 1.14$\times 10^{-6}$ / 1.00 && 822.75 (  6 / 1957) / 3.72  (0.02 / 8.91)  / 1.78$\times 10^{-6}$ / 0.16  \\ \hline
$( 4,5, 5)$ & 6.74 (  3 / 108) / 0.03  (0.00 / 0.48)  / 1.07$\times 10^{-6}$ / 1.00 && 11.11 (  4 / 159) / 0.05  (0.02 / 0.70)  / 7.73$\times 10^{-7}$ / 0.83 \\
 $( 4,5,10)$ & 4.42 (  3 /   5) / 0.02  (0.00 / 0.03)  / 1.04$\times 10^{-6}$ / 1.00 && 36.66 (  5 / 429) / 0.17  (0.02 / 1.97)  / 1.86$\times 10^{-6}$ / 0.67 \\
 $( 4,5,20)$ & 4.27 (  4 /   5) / 0.04  (0.03 / 0.06)  / 9.58$\times 10^{-7}$ / 1.00 && 171.83 ( 11 / 1255) / 1.78  (0.11 / 13.14)  / 1.93$\times 10^{-6}$ / 0.83 \\ \hline
$( 4,6, 5)$ & 4.62 (  3 /   6) / 0.02  (0.02 / 0.03)  / 8.57$\times 10^{-7}$ / 1.00 && 23.40 (  5 / 542) / 0.11  (0.02 / 2.45)  / 1.74$\times 10^{-6}$ / 0.52 \\
 $( 4,6,10)$ & 4.35 (  3 /   5) / 0.02  (0.02 / 0.03)  / 1.10$\times 10^{-6}$ / 1.00 && 114.85 (  7 / 1686) / 0.67  (0.03 / 9.81)  / 1.82$\times 10^{-6}$ / 0.52 \\
 $( 4,6,15)$ & 4.29 (  3 /   5) / 0.14  (0.09 / 0.17)  / 1.29$\times 10^{-6}$ / 1.00 && 330.00 ( 11 / 1859) / 11.15  (0.37 / 58.41)  / 2.03$\times 10^{-6}$ / 0.58 \\ \hline
$( 5,6, 5)$ & 5.00 (  3 /  51) / 0.03  (0.02 / 0.23)  / 1.53$\times 10^{-6}$ / 1.00 && 34.42 (  5 / 206) / 0.16  (0.02 / 0.97)  / 1.32$\times 10^{-6}$ / 0.89 \\
 $( 5,6,10)$ & 4.82 (  3 /  53) / 0.03  (0.02 / 0.33)  / 9.36$\times 10^{-7}$ / 1.00 && 74.97 (  7 / 473) / 0.44  (0.03 / 2.82)  / 1.79$\times 10^{-6}$ / 0.60 \\
 $( 5,6,15)$ & 4.34 (  3 /   5) / 0.14  (0.09 / 0.17)  / 1.07$\times 10^{-6}$ / 1.00 && 100.04 ( 24 / 647) / 3.59  (0.84 / 23.34)  / 1.54$\times 10^{-6}$ / 0.55  \\
\bottomrule
\end{tabular*}
\end{center}
\end{sidewaystable}

\setlength\rotFPtop{0pt plus 1fil}
\begin{sidewaystable}
\begin{center}
\caption{Computational results for the cases $p\neq q$ with (iii) ${\mathcal A}$ is a general tensor and ${\mathcal B}$ is an ${\mathcal M}$-tensor, and (iv) $({\mathcal A},{\mathcal B})$ are general tensors}\vskip 0.2mm
\label{table4}
\small
\def\temptablewidth{1\textwidth}
\begin{tabular*}{\temptablewidth}{@{\extracolsep{\fill}}llll}\toprule
& (iii) ${\mathcal A}$ is a general tensor and ${\mathcal B}$ is an ${\mathcal M}$-tensor && (iv) $({\mathcal A},{\mathcal B})$ are general tensors\\\cline{2-2} \cline{4-4}
$(p,q,n)$ & Iter. ($k_{\min}$ / $k_{\max}$) / Time ($t_{\min}$ / $t_{\max}$) / Err / SR && Iter. ($k_{\min}$ / $k_{\max}$) / Time ($t_{\min}$ / $t_{\max}$) / Err / SR \\ \midrule
$( 4,3, 5)$ & 24.57 (  5 / 146) / 0.11  (0.02 / 0.64)  / 1.33$\times 10^{-6}$ / 0.95 && 23.38 (  4 / 180) / 0.10  (0.02 / 0.78)  / 1.04$\times 10^{-6}$ / 1.00 \\
 $( 4,3,10)$ & 43.85 (  7 / 242) / 0.19  (0.03 / 1.06)  / 6.69$\times 10^{-7}$ / 0.99 && 53.00 (  9 / 247) / 0.23  (0.05 / 1.08)  / 1.20$\times 10^{-6}$ / 1.00 \\
 $( 4,3,20)$ & 155.53 ( 12 / 1496) / 0.70  (0.05 / 6.72)  / 1.32$\times 10^{-6}$ / 0.99 && 178.52 ( 18 / 788) / 0.80  (0.08 / 3.49)  / 9.97$\times 10^{-7}$ / 1.00 \\ \hline
$( 5,4, 5)$ & 19.47 (  5 / 145) / 0.09  (0.03 / 0.66)  / 1.30$\times 10^{-6}$ / 0.99 && 38.46 (  6 / 409) / 0.18  (0.02 / 1.86)  / 1.24$\times 10^{-6}$ / 1.00 \\
 $( 5,4,10)$ & 48.07 ( 10 / 347) / 0.22  (0.05 / 1.59)  / 1.06$\times 10^{-6}$ / 0.97 && 69.40 (  9 / 237) / 0.32  (0.05 / 1.09)  / 1.40$\times 10^{-6}$ / 0.99 \\
 $( 5,4,20)$ & 115.90 ( 16 / 704) / 1.20  (0.17 / 7.43)  / 9.66$\times 10^{-7}$ / 0.99 && 282.64 ( 16 / 1824) / 2.93  (0.16 / 18.77)  / 1.69$\times 10^{-6}$ / 1.00 \\ \hline
$( 6,4, 5)$ & 28.68 (  6 / 166) / 0.14  (0.03 / 0.75)  / 1.09$\times 10^{-6}$ / 0.95 && 39.70 (  6 / 234) / 0.19  (0.03 / 1.09)  / 1.17$\times 10^{-6}$ / 1.00 \\
 $( 6,4,10)$ & 68.68 ( 10 / 378) / 0.39  (0.05 / 2.17)  / 1.00$\times 10^{-6}$ / 0.97 && 73.16 ( 11 / 306) / 0.42  (0.06 / 1.76)  / 1.25$\times 10^{-6}$ / 1.00 \\
 $( 6,4,15)$ & 156.38 ( 13 / 894) / 5.39  (0.44 / 30.92)  / 1.23$\times 10^{-6}$ / 0.99 && 170.83 ( 13 / 679) / 5.90  (0.44 / 23.45)  / 1.49$\times 10^{-6}$ / 1.00 \\\hline
$( 6,5, 5)$ & 32.94 (  7 / 250) / 0.16  (0.03 / 1.20)  / 1.20$\times 10^{-6}$ / 0.97 && 33.02 (  6 / 166) / 0.16  (0.03 / 0.78)  / 1.33$\times 10^{-6}$ / 1.00 \\
 $( 6,5,10)$ & 68.08 ( 11 / 427) / 0.41  (0.06 / 2.54)  / 1.14$\times 10^{-6}$ / 0.98 && 78.88 ( 16 / 623) / 0.47  (0.09 / 3.71)  / 1.74$\times 10^{-6}$ / 1.00 \\
 $( 6,5,15)$ &  113.35 ( 17 / 629) / 4.07  (0.58 / 22.74)  / 1.18$\times 10^{-6}$ / 0.96 && 200.24 ( 11 / 1016) / 7.20  (0.37 / 36.77)  / 1.11$\times 10^{-6}$ / 1.00 \\ \midrule
$( 3,4, 5)$ & 26.18 (  3 / 304) / 0.11  (0.02 / 1.33)  / 1.03$\times 10^{-6}$ / 0.92 && 18.13 (  4 / 151) / 0.08  (0.02 / 0.67)  / 1.21$\times 10^{-6}$ / 0.95 \\
 $( 3,4,10)$ & 62.61 (  3 / 694) / 0.28  (0.02 / 3.06)  / 1.11$\times 10^{-6}$ / 0.92 && 112.84 (  8 / 633) / 0.49  (0.03 / 2.79)  / 1.23$\times 10^{-6}$ / 0.79 \\
 $( 3,4,20)$ & 129.33 (  4 / 1074) / 0.58  (0.02 / 4.79)  / 1.37$\times 10^{-6}$ / 0.91 && 1099.80 (517 / 1789) / 4.97  (2.34 / 8.14)  / 1.58$\times 10^{-6}$ / 0.10 \\ \hline
$( 4,5, 5)$ & 33.05 (  5 / 513) / 0.15  (0.02 / 2.31)  / 1.52$\times 10^{-6}$ / 0.94 && 34.04 (  5 / 544) / 0.15  (0.02 / 2.45)  / 1.38$\times 10^{-6}$ / 0.98 \\
 $( 4,5,10)$ & 78.41 (  9 / 440) / 0.36  (0.05 / 2.03)  / 1.72$\times 10^{-6}$ / 0.80 && 142.97 (  9 / 595) / 0.66  (0.05 / 2.71)  / 1.10$\times 10^{-6}$ / 0.86 \\
 $( 4,5,20)$ & 238.63 ( 11 / 1749) / 2.47  (0.11 / 18.13)  / 1.06$\times 10^{-6}$ / 0.93 && 914.82 ( 70 / 1990) / 9.45  (0.73 / 20.56)  / 5.25$\times 10^{-7}$ / 0.17 \\ \hline
$( 4,6, 5)$ & 33.65 (  3 / 415) / 0.16  (0.02 / 1.89)  / 1.38$\times 10^{-6}$ / 0.92 && 54.63 (  5 / 823) / 0.26  (0.02 / 3.74)  / 1.38$\times 10^{-6}$ / 0.93 \\
 $( 4,6,10)$ & 100.46 (  4 / 1048) / 0.58  (0.02 / 5.97)  / 1.40$\times 10^{-6}$ / 0.85 && 180.58 ( 11 / 753) / 1.05  (0.06 / 4.38)  / 1.19$\times 10^{-6}$ / 0.62 \\
 $( 4,6,15)$ & 189.75 (  4 / 1777) / 6.52  (0.12 / 61.34)  / 1.66$\times 10^{-6}$ / 0.79 && 598.71 ( 10 / 1983) / 20.64  (0.33 / 68.38)  / 4.80$\times 10^{-7}$ / 0.42 \\ \hline
$( 5,6, 5)$ & 36.48 (  5 / 400) / 0.17  (0.03 / 1.83)  / 1.32$\times 10^{-6}$ / 0.96 && 30.97 (  5 / 215) / 0.15  (0.03 / 1.00)  / 1.60$\times 10^{-6}$ / 1.00 \\
 $( 5,6,10)$ & 86.46 (  8 / 581) / 0.51  (0.05 / 3.40)  / 1.77$\times 10^{-6}$ / 0.91 && 135.49 (  5 / 744) / 0.81  (0.03 / 4.37)  / 1.28$\times 10^{-6}$ / 0.92 \\
 $( 5,6,15)$ &151.18 ( 11 / 942) / 5.43  (0.39 / 33.79)  / 1.45$\times 10^{-6}$ / 0.89 && 617.83 ( 18 / 1929) / 22.23  (0.62 / 69.53)  / 9.68$\times 10^{-7}$ / 0.81 \\ \bottomrule
\end{tabular*}
\end{center}
\end{sidewaystable}

As we have mentioned above, although our solutions existence theorems are established for the case $p\geq q$, the proposed generalized Newton method does not rely on such a relation $p\geq q$. Therefore, in Tables \ref{table3} and \ref{table4}, we correspondingly consider the two cases $p>q$ and $p<q$ with the four scenarios on tensors. It can be seen from the results that the generalized Newton method performs well for the case $p>q$, especially for the case with two general tensors ${\mathcal A}$ and ${\mathcal B}$ (see Table \ref{table4}). When dealing with the case $p<q$, the best performance of the generalized Newton method corresponds to the scenario that both ${\mathcal A}$ and ${\mathcal B}$ are ${\mathcal M}$-tensors.

From all the data reported in this section, it is not difficult to see that the generalized Newton method is a reliable solver for most of TAVEs. Here, we shall notice that, for the failure cases, the generalized Newton method \eqref{GNewton} can successfully find a solution to TAVEs if the starting point $x_0$ is sufficiently near the solution. It means that the starting point would affect the performance of the proposed method for TAVEs. However, we completely do not know where is the solution for real-world problems. So, we use the aforementioned constant starting point $x_0$ throughout the experiments for the purpose of investigating the real performance of \eqref{GNewton} on TAVEs.  Meanwhile, one question raised is that can we design an algorithm which is independent on initial points? We would like to leave it as our future work.

\section{Conclusions}\label{Conclusion}
In this paper, we considered the system of TAVEs, which is an interesting generalization of the classical absolute value equations in the matrix case. By the employment of degree theory, we showed that the solutions set of the system of TAVEs with $p>q$ is nonempty and compact. Moreover, by the utility of fixed point theory, we proved that TAVEs with $p=q$ has at least one solution under some checkable conditions. However, we did not give the answer when such a problem has a unique solution for the case where ${\mathcal B}$ is not a negative unit tensor. Moreover, what will happen when we consider the case where TAVEs with the setting of $p<q$? In the future, we would like to try to answer these questions. On the other hand, our numerical results show that the generalized Newton method performs well in many cases. However, it still fails in some cases. So, can we design structure-exploiting algorithms which are independent on the starting point? This is also one of our future concerns.

\medskip
\begin{acknowledgements}
C. Ling and H. He were supported in part by National Natural Science Foundation of China (Nos. 11571087 and 11771113) and Natural Science Foundation of Zhejiang Province (LY17A010028). L. Qi was supported by the Hong Kong Research Grant Council (Grant Nos. PolyU 15302114, 15300715, 15301716 and 15300717).
\end{acknowledgements}


\end{document}